\documentclass[12pt]{amsart}
\usepackage{times}       
 
\newcommand{\cf}{{\it cf.} }
\newcommand{\ie}{{\it i.e.} }
\newcommand{\eg}{{\it e.g.} }

\newcommand{\F}{\mathbb{F}}

\newcommand{\Q}{\mathbb{Q}}

\newcommand{\C}{\mathbb{C}}
  \newcommand{\Z}{\mathbb{Z}}

\renewcommand{\P}{\mathbb{P}}

\newcommand{\sG}{{\mathcal{G}}}

\newcommand{\sT}{{\mathcal{T}}}

\newcommand{\sX}{{\mathcal{X}}}

\newcommand{\inj}{\hookrightarrow}

\newcommand{\surj}{\rightarrow\!\!\!\!\!\rightarrow}

\newcommand{\Ker}{\operatorname{Ker}}

 \newcommand{\Char}{\operatorname{char}}
\newcounter{spec}

\swapnumbers

\newtheorem{thm}{Theorem}[section]
\newtheorem{lemma}[thm]{Lemma}

\newtheorem{cor}[thm]{Corollary}
\newtheorem{conj}[thm]{Conjecture}

\theoremstyle{definition}
\newtheorem{defn}[thm]{Definition}

\newtheorem{ex}[thm]{Example}
\newtheorem{exs}[thm]{Examples}
\newtheorem{qn}[thm]{Question}

\newtheorem{rem}[thm]{Remark}
\newtheorem{rems}[thm]{Remarks}

\numberwithin{equation}{section}

 at10pt
 at12pt

\begin{document}

 \title{On the observability of Galois representations and the Tate conjecture.}
 
\author{Yves Andr\'e}
 
   \address{ }
\email{ }
 \keywords{observability, semisimplicity, Tate conjecture, $p$-adic periods}
\subjclass{11G, 14F, 14L}
   
 \begin{sloppypar}

  \begin{abstract}  The Tate conjecture has two parts: i) Tate classes are linear combination of algebraic classes, ii) semisimplicity of Galois representations (for smooth projective varieties).
 B. Moonen proved that i) implies ii) in characteristic 0, using $p$-adic Hodge theory. 

We show that an unconditional result lies behind this implication: the {\it observability} of arithmetic monodromy groups of geometric origin (in any characteristic) - which leads to a sharpening of Moonen's result. 

We also discuss another aspect of the Tate conjecture related to the transcendence of $p$-adic periods.
   \end{abstract}
 \maketitle
  
   \section*{Introduction}
  
Let $K$ be a field of characteristic $p\geq 0$, 
 $K^s$ a separable closure of $K$ and $G_K= Gal(K^s/K)$ the absolute Galois group. 

\medskip Let $X$ be a smooth projective $K$-variety. For $\ell \neq p $, and any $j\in \Z$, the twisted $\ell$-adic etale cohomology $H^\ast(\bar X, \Q_\ell(j))$ of $\bar X = X_{K^s}$ is a continuous representation of $G_K$. 

 Let us denote by $\sT^\ast(X)$ the graded $\Q_\ell$-space of {\it Tate classes}, \ie $G_K$-invariants in $\oplus\, H^{2j}(\bar X, \Q_\ell(j))$.
 
 The Tate conjecture has two parts\footnote{see \eg \cite{To17} for some historical background}:
 
 \begin{conj}[Tate] Assume that $K$ is finitely generated over its prime field. Then:
 
 {   $S(X):$  \, $H^\ast(\bar X, \Q_\ell)$ is a semisimple $G_K$-representation.}
          
         {  $T(X):$ \, the cycle class map $CH^\ast(X)\otimes \Q_\ell \rightarrow \sT^\ast(X)$ is surjective.\,}\end{conj}

($S(X)$ is also attributed to Grothendieck-Serre, and is equivalent to: the Zariski closure $H_\ell$ of the image of $G_K$ in $GL(H^\ast(\bar X, \Q_\ell))$ is reductive.)
 
 \begin{qn} What is the relation between $T(X)$ (invariants) and $S(X)$ (semisimplicity)? \end{qn}

 A surprising partial answer was recently given by B. Moonen \cite{M19}: 

 \smallskip {\it if $p=0,\,$   $T(X)$ for all smooth projective varieties $X$ over all finite extensions of $K$ implies $S(X)$ for all such $X$.\,}

  \medskip In this note, we show that behind this result, there is an inconditional fact, which is valid in any characteristic (Theorem \ref{T2}): 
  
  \smallskip {\it The Zariski closure $H_\ell$ of the image of $G_K$ in $GL(H^\ast(\bar X, \Q_\ell))$ is the largest closed subgroup which fixes $\sT^\ast(X^n)$ for all $n$. } 
  
 \smallskip  This can be translated into: {\it $H_\ell$ is an {\it observable} subgroup of $GL(H^\ast(\bar X, \Q_\ell))$} (in the sense of Bialynicki-Birula, Hochschild and Mostow). We actually use the theory of observability to prove the above fact - which holds in greater generality: one may replace $X$ by {\it any mixed motive} over $K$ (in the sense of Nori). 
  
  On the other hand, when $p=0$, one has an inconditional motivic Galois theory, which provides reductive motivic Galois groups in the pure case (in various guises, but all equivalent). It is then easy to deduce from the previous theorem a refined version of Moonen's theorem (Corollary \ref{cor1}): 
  
 \smallskip \centerline{\it if $p=0,\,$ $T(X^n)$ for all  $n\,$ implies $S(X)$.} 
  
    \medskip In the last section, we consider another aspect of the Tate conjecture when $K$ is a number field, which is related to the $p$-adic periods (depending on a chosen embedding $\iota_p$ of $\bar K$ in $\bar \Q_p$). Clarifying a result of \cite{A90} in terms of motivic Galois groups, we establish a 
    
    {\it $p$-adic analog of the Grothendieck period conjecture for $X$ if $T(X^n)$ holds for all  $n\,$ and $\iota_p$ is sufficienty general}.
    
    \medskip The use of $p$-adic Hodge theory somehow relates that section to the previous one (in the proof of Moonen's theorem as well as in the proof of Theorem \ref{T2}, the main ingredient is that $p$-adic representations coming from geometry are Hodge-Tate).

 \section{Observability}
   
 Let  $G $ be a linear algebraic group  (= a smooth affine group scheme of finite type) 
 over a field $F$, and  
  let $H\subset G$ be a closed linear algebraic subgroup. 
   
  By ``representation" $V$ of $H$ or $G$, we mean a finite-dimensional rational $F$-linear representation. We denote by $V^H$ the subspace of $V$ where $H$ acts trivially.
    
    \begin{defn}[Bialynicki-Birula, Hochschild, Mostow] 
    {  $H$ is} {\it observable} {in $G\,$ if every representation of $H$ is contained in a representation of $G$.} 
    \end{defn}

 \begin{thm}[Hochschild, Mostow, Grosshans]\cite[1.10]{Gr97}\label{T1}  The following conditions are equivalent:
 
 \smallskip  { $1)\;$ $H$ is observable in $G$,}
          
 \smallskip  { $2)\;$  every character of $H$ occurs\footnote{\ie is contained} in a representation of $G$,} 
 
 \smallskip  { $3)\;$  for every character of $H$ occuring in a representation of $G$, its inverse also occurs in a representation of $G$,} 
  
 \smallskip  { $4)\;$  $H$ is the isotropy group of some element in some representation of $G$,}
   
  \smallskip  { $5)\;$  $G/H$ is quasi-affine (i.e. open in some affine $F$-variety),}
  
  \smallskip  { $6)\;$  $H^0$ is observable in $G^0$.}
  \end{thm}
     
\begin{rem}  
 Most expositions of this theorem (\eg \cite{Gr97}) assume that $F$ is algebraically closed, but a close look at the proofs shows that this assumption is unnecessary. 
  \cite{TB05} deals specifically with observability over non-algebraically closed fields and proves that $H$ is observable in $G$ if and only if $H_{K'}$ is observable in $G_{K'}$ for any extension $K'/K$ - but, strangely,  does not deal with characters. Since item $3)$ is fundamental in the sequel, we give a short and elementary proof that $3)\Rightarrow 4)$ (since $4) \Rightarrow 1)$ is detailed in \cite{TB05} and $1) \Rightarrow 2) + 3)$ is trivial, this is the only point to be justified for $2)$ and $3)$ over a general field $F$). 
 
 By Chevalley's theorem, $H$ is the stabilizer of a line $L$ in a representation $V$ of $G$. If, according to $3)$, the dual $L^\vee$, viewed as a representation of $H$, is contained in a $G$-representation $W$, then $L\otimes L^\vee$ is contained in the representation $V\otimes W$ and any non-zero element $\ell\otimes \ell'$ is fixed by $H$. Let $H'\supset H$ be the isotropy group of $\ell\otimes \ell'$. Since the Segre map $\P(V)\times \P(W)\to \P(V\otimes W)$ is an embedding,  $H'$ acts by homotheties on $\ell$ and $\ell'$, hence $H=H'$, which gives $4)$.  
     \end{rem}  

 \begin{exs}  $i)$ Any normal subgroup is observable (by condition $5)$).
 Conversely, if $H$ is observable in $G$ and if for any $G$-representation $V$, 
 $V^H$ is $G$-stable, then $H$ is normal in $G\,$  (Esnault-Hai-Sun \cite{EHS}; \cf also \cite[C.3]{A21}).

 \medskip $ii)$ A Borel subgroup $B$ of a connected (quasi-split) reductive group $G$ is not observable  unless $G$ is a torus (by condition $5)$ again). 
 
As a special case, the standard $SL_2$-representation $V$ becomes, by restriction to the subgroup $B$ of upper triangular matrices, an extension of a character by its inverse; the first one is not contained in a $SL_2$-representation, so that condition $3)$ is violated.

 \medskip $iii)$ Condition $3)$ is automatic for characters of finite order.
  \end{exs}
 
Observable subgroups are those ``detected" by Invariant Theory, whence the name:  

\begin{lemma}\label{L1} If $H$ is observable in $G$ and if for any $G$-representation $V$, one has $V^H = V^G$, then $H=G\,$ (the converse is obvious).\end{lemma} 
 
 \begin{proof}  ``$V^H= V^G$ for all $V$" means that $Rep\, G \rightarrow Rep\,H$ is fully faithful. ``$H$ observable" means that any $H$-representation $W$ is contained in a $G$-representation $V$, and dually is a quotient of a $G$-representation $V'$. By full faithfulness, the composition $V'\surj W \inj V$ is a $G$-morphism, thus $W$ is a $G$-representation. Hence $Rep\, G \rightarrow Rep\,H$ is essentially surjective. \end{proof}
   
 \begin{rems} $i)$ The condition $V^H = V^G$ is satisfied for a Borel subgroup $H$ of a reductive group $G$. 
 
\smallskip $ii)$  Traditionally, observability theory is not based on the tannakian viewpoint, but on the old ``adjunction/transfer principle" (valid for any closed subgroup $H\subset G$):
  $$(F[G/H]\otimes V)^G= V^H,$$
  where the action $G$ on $F[G/H]$ is by left translation. 
  
  However these two points of view are essentially equivalent, replacing $F[G]$ by the collection of representations of $G$ or, if $G$ is reductive, by the collection of mixed tensor spaces $V^{\otimes m} \otimes V^{\vee \,\otimes n}$ on a faithful $G$-representation $V$  (since $G$-repre\-sentations are contained in direct sums of mixed tensor spaces built on $V$).\end{rems}

 \begin{lemma}\label{L2} Let $V$ be a faithful representation of $G$, and let $H'$ be the largest closed subgroup of $G$ which fixes the $H$-invariants in every $G$-representation. Then $H'$ is the smallest observable subgroup of $G$ containing $H$.  In particular, $H$ is observable in $G$ if and only if $H=H'$.\end{lemma}
 
 \begin{proof} Clearly $H\subset H'$. By the previous lemma, it suffices to show that $H'$ is observable\footnote{if $H''\subset H'$ is another observable subgroup (in $G$ hence in $H'$) containing $H$, the previous lemma applied to $H''\subset H'$ shows that $H''=H'$}.
Note that $H'$ is the closed subgroup of $G$ of elements fixing the $H$-invariant functions $f\in F[G]$. Since $H'$ is of finite type, finitely many functions $f_i$ suffice,  $H'$ is the isotropy group of $\oplus f_i$ (in a direct sum of subrepresentations of $F[G]$), and one concludes by condition $4)$. \end{proof}

\section{Observability and the Tate conjecture.} 

\subsection{Observability of Galois representations.} Let us come back to the initial situation, where $X$ is a smooth projective variety over a finitely generated field $K$.
 
  \smallskip Let us set $G_\ell  = GL(H^\ast(\bar X, \Q_\ell))$ and denote by $H_\ell\subset G_\ell$ the Zariski closure of the image of $G_K$.
   
  \medskip
 \begin{thm}\label{T2}  
  { $H_\ell $ is observable in $G_\ell $.  Equivalently\footnote{by Lemma \ref{L2} and the remark just before it}, $H_\ell$ is the largest subgroup of $G_\ell$ which fixes $\sT^\ast(X^n)$ for all $n$.}\end{thm}
  
 \begin{proof} I - The case $p = 0$.

  \smallskip
$a)$ ({\it Reduction to the case $K = \Q$}). By a standard specialization argument due to Serre \cite[\S 6.4]{S94}\cite[p. 149]{S97}, using the fact that number fields are Hilbertian, one reduces to the case where $X$ is defined over a number field $K$, which we may assume to be Galois over $\Q$ (the groups $H_\ell$ and $G_\ell$ being unchanged). 

Let $X'$ be $X$ considered as a $\Q$-scheme, and $H'_\ell \subset G'_\ell$ the corresponding groups for $X'/\Q$. The $G_\Q$-representation $H^\ast(\bar X', \Q_\ell)$ is induced from the $G_K$-representation $V_\ell = H^\ast(\bar X, \Q_\ell)$. Therefore its restriction to $G_K$ is the direct sum of conjugates of $V_\ell$, one may view $G_\ell$ as a subgroup of $G'_\ell$, $H^0_\ell$ as the image of $H'^0_\ell$ in $G_\ell$, and $G_\ell/H^0_\ell$ as a closed subspace of $G'_\ell/H'^0_\ell$. By conditions $5)$ and $6)$ in Theorem \ref{T1},  one concludes that $H_\ell$ is observable in $G_\ell$ if $H'_\ell$ is observable in $G'_\ell$\footnote{alternatively: if $H'_\ell$ is observable in $G'_\ell$, it is observable in the subgroup $G''_\ell := \Pi GL(V_\ell^\sigma)$. Since $G_\ell $ is a quotient of $G''_\ell$, and $H^0_\ell$ is the image of $H'^0_\ell$, one may conclude using conditions $2)$ and $6)$ in Theorem \ref{T1}}.
   
 \smallskip
 $b)$ {\it(Moonen's argument)}.  A character of $H_\ell$ contained in a $G_\ell$-representation $V$  comes from a character  $\chi  : G_\Q \rightarrow \Z_\ell^\times$.   
 
   One may assume that $\chi$ is infinite. By {\it Kronecker-Weber}, the extension $L/\Q$ corresponding to $\Ker\, \chi$ contains the cyclotomic $\Z_\ell$-extension $\Q_\infty$ with finite index. As $\ell$ is totally ramified in $\Q_\infty$, the inertia $I\subset Gal(L/\Q)$ at a prime above $\ell$ is open.  
 
  Since $V$ is of geometric origin, it is {\it Hodge-Tate} \cite{F}\footnote{this condition passes to Galois subrepresentations}; hence so is $\chi$, and there is an open subgroup $J\subset I$ such that $\chi_{\mid J}$ is an integral power of the cyclotomic character \cite[3.9]{Fo}. 
 
 \smallskip $c)$ ({\it Conclusion}). Therefore $\chi$ is potentially a Tate twist, \ie the realization of a power of the Tate motive. Such a Tate twist is a $G_\ell$-representation, and one may assume after twisting $V$ that $\chi$ is of finite order. One concludes by condition $3)$ in Theorem \ref{T1}. 

  One also concludes (by condition $2)$) that {\it any character of $H_\ell$} (not supposed a priori to be contained in a $G_\ell$-representation) {\it is potentially a Tate twist}.\end{proof}
    
  \begin{rem} Reduction to $K=\Q$ is essential in steps $b)$ and $c)$: 
    if $X$ is an elliptic curve with complex multiplication in which $\ell$ splits, $H^1(\bar X, \Q_\ell)$ is the direct sum of two one-dimensional $H_\ell$-representations which are not potentially Tate twists. \end{rem}

\begin{proof} II - The case $p >0$.

If $K=\F_q$ is finite, the functional equation of the zeta function of $X$ shows that if $\alpha$ is a Frobenius eigenvalue, so is $q^n/\alpha$ for some $n$. Hence if a character of $H_\ell$ occurs in a $G_\ell$-representation, so does its inverse. Therefore $H_\ell$ is observable in $G_\ell$ in this case. 

\smallskip
  If $K$ has transcendence degree $>0$, we give two proofs. 

  \medskip
1) The first proof parallels the case $p=0$.  

\smallskip $a')$ A specialization argument \`a la Serre, using the fact that function fields are Hilbertian, reduces to the case of transcendence degree $1$. One further reduces to the case $K= {\bf F}_q(T)$ as in step $a)$ above. 

\smallskip $b')$ One emulates Moonen's argument, replacing
Kronecker-Weber by {\it Hayes' theorem} \cite{H94}, which describes abelian extensions of $K= {\bf F}_q(T)$: they are compositions of 

- Carlitz-cyclotomic extensions associated to closed points of ${\bf P}^1_{{\bf F}_q}$, and 

- extensions of $K$ which come from extensions of ${\bf F}_q$. 

  \smallskip\noindent (Under Carlitz' representation: ${\bf F}_q[T] \rightarrow End_{ab.gp}(\bar K)$,
 $\; P \mapsto $ separable polynomial $\tilde P$ with coefficients in ${\bf F}_q[T]$,  
the roots of $\tilde P\,$ (\ie the $P$-torsion points of $\bar K$) generate an abelian extension of $K = {\bf F}_q(T)$.)

\smallskip 
Carlitz-cyclotomic extensions, when infinite cyclic, have infinite wild ramification at a suitable prime. 
In our situation, $Gal(L/K) = Im\, \chi \subset \Z_\ell^\times$, hence $L$ is potentially cyclic and tamely ramified everywhere. Therefore $L$ comes from a representation of the absolute Galois group of a finite extension of ${\bf F}_q$. 

\smallskip $c')$ By suitable specialization of $T$, one then reduces to case when $K$ is finite.  $\square$

\medskip 2) For the second proof, one views $K$ as the function field of some variety $S/{\bf F}_q$ and $X$ as the generic fiber of $\sX/S$, and uses the fact that characters of the geometric monodromy are of finite order (Deligne \cite{D80}). By condition $6)$ in Theorem \ref{T1}, one may replace $K$ be a finite separable extension, hence assume that the geometric monodromy group $H_{geom}$ has no non-trivial character, and that $S$ has an $\F_q$-rational point $s$ at which $X$ has good reduction. If $L$ is a $1$-dimensional $H_\ell$-representation  contained in a $G_\ell$-representation, $H_{geom}$ acts trivially on $L$, therefore $L$ comes from a representation of the Galois group of ${\bf F}_q$ and by specialization at $s$, one reduces to the case when $K$ is finite\footnote{one finds an argument of similar flavor in \cite[T. 7]{K04}}. \end{proof}
 
  \begin{rems} $i)$ The theorem extends, with the same proof, to Galois representations attached to Nori  {\it mixed motives} ($p$-adic realizations of mixed motives are discussed in detail in \cite[\S 4]{DN}\footnote{however we are not aware of a written exposition of the theory of Nori motives in positive characteristic}).
  
  $ii)$ In \cite{A21}, observability was used in the context of variations of mixed Hodge structures instead of Galois representations. 
  \end{rems} 
 
\subsection{A refinement of Moonen's theorem.} 
  \begin{cor}\label{cor1} 
 {  Assume that $\Char K =0$. Then $T(X^n)$ for all  $n\,$ implies $S(X)$.\,} \end{cor}

\begin{proof} Since $\Char K =0$, we have a tannakian semisimple $\otimes$-category of pure motives (with coefficients in $F= \Q_\ell$): it is defined, equivalently, using motivated cycles or using Nori's construction. In particular, we have a reductive subgroup $G_{mot, \ell}(X) \subset G_\ell = GL(H^\ast(\bar X, \Q_\ell))$, the {\it motivic Galois group} of $X$ (\cf \cite{A96}, and \cite{Ar13} for the equivalence with the setting of pure Nori motives). For any $G_\ell$-representation $V$, $V^{G_{mot, \ell}}$ consists of motivated classes. They are Tate classes, and algebraic classes\footnote{by which we understand $\Q_\ell$-linear combinations of fundamental classes of closed subvarieties, \ie elements of the image of the cycle class map $CH^\ast(X)\otimes \Q_\ell \rightarrow \sT^\ast(X)$.} are motivated. 

Therefore, 
 $H_\ell \subset G_{mot, \ell}$, and $T(X^n)$ for all $n$ implies $V^{H_\ell} = V^{G_{mot, \ell}}$. 

\smallskip Since $H_\ell$ is observable in $G_{mot, \ell}$ (being observable in the bigger group $G_\ell$ by Theorem \ref{T2}), one concludes that $H_\ell = G_{mot, \ell}$ by Lemma \ref{L1}. Since $G_{mot, \ell}$ is reductive, this gives $S(X)$. \end{proof}

 \begin{rem}
 If  $p >0$, a recent result by O'Sullivan provides a group with similar properties as $G_{mot, \ell}$ in characteristic $0$. Its invariants are ``fractionally algebraic classes", \ie classes $\xi$ in $\sT^\ast(X)$ for which there exists a projective smooth $Y/K$ and an algebraic class $\eta\in  \sT^\ast(Y)$ such that $ \xi\otimes \eta \in \sT^\ast(X \times Y)$ is algebraic (\cf \cite{O'S} and the forthcoming sequel). However, it is not known to be reductive, and one cannot derive as above the semisimplicity of the corresponding Galois representations.
\end{rem}

\medskip
 \section{Another remark on the Tate conjecture} 
 
 \subsection{Period isomorphism: the complex case.} We retain the same setting, but restrict to the case when $K$ is a number field ($\bar K$ being its algebraic closure in $\C$). 
  
  Instead of the etale cohomology of $X_{\bar K}$, we consider the De Rham cohomology $H^\ast_{dR}(X)$ and the Betti cohomology $H^\ast_B(X)$, which are related by  
 the period isomorphism
 $$\varpi: \; H^{\ast}_{dR}(X)(j)\otimes_K \C \cong H^{\ast}_B(X)(j)\otimes_\Q \C.$$  

There is a parallelism, which is expounded in \cite[Ch. 7]{A04}, between Tate's conjecture and a weak form of Grothendieck's period conjecture, according to which the analogs of Tate classes are what may be called {\it De Rham-Betti classes}\footnote{or {\it Grothendieck classes}, in view of the Grothendieck period conjecture below}, \ie classes $\xi \in  \oplus H^{2j}_{dR}(X)(j)$ such that $\varpi(\xi)\in \oplus H^{2j}_B(X)(j)$ (they have also been investigated in \cite{BC16} and elsewhere). 

 Let us denote by $\sG^\ast(X)$ the graded $\Q$-space of {\it De Rham-Betti classes}.

\begin{conj}[Grothendieck - weak form]
  $\,$
  
     \smallskip     $G(X):\,$ The cycle class map $CH^\ast(X) \rightarrow \sG^\ast(X)$ is surjective.  \end{conj}

          The analogy with $T(X)$ is patent. In one of its stronger forms\footnote{a detailed discussion, including a historical viewpoint, can be found in \cite{A20}. At the time of the first expositions of Grothendieck's period conjecture, a non-conjectural motivic Galois theory was not available, which made them a little awkward (or look like ``meta-conjectures".}, Grothendieck's period conjecture predicts 
 
\smallskip
\centerline{ $  tr. deg_\Q$ (periods of $X$) $= dim\, G_{mot}(X).$ } 
 
 \begin{rem} One defines naturally a category $Vec_{K,\Q}$ whose objects are pairs of finite dimensional vector spaces $(W, V)$ over $K$ and $\Q$ respectively, together with an isomorphism $\iota: W\otimes_K \C \cong V\otimes_\Q \C$. This is a tannakian category over $\Q$, and $(W,V,\iota)\mapsto V$ yields a fiber functor. 
 
 The De Rham and Betti realizations and the period isomorphism yield an exact $\otimes$-functor from the tannakian category of motives (say in the sense of Nori) to $Vec_{K,\Q}$, whence a morphism in the other direction between the tannakian groups, which expresses the tannakian situation behind $G(X)$. The issue of {\it observability} is quite relevant in this context.
   \end{rem}

 \subsection{Period isomorphism: the $p$-adic case.}
 
Let us fix an embedding
 $\iota_p: \bar K \subset \bar \Q_p$, and denote by $\hat K$ the $p$-adic completion of $K$.

Artin's comparison theorem between Betti and etale cohomologies, combined with the $p$-adic-etale comparison theorem (Fontaine, Messing, Faltings et al.) gives rise to a $p$-adic analog of $\varpi$:
 $$ \varpi_{\iota_p}: \; H^{\ast}_{dR}(X)(j)\otimes_K B_{dR, \hat K} \cong H^{\ast}_B(X)(j)\otimes_\Q B_{dR, \hat K}.$$  
One may define similarly the {\it $p$-adic De Rham-Betti classes}, \ie classes $\xi \in  \oplus H^{2j}_{dR}(X)(j)$ such that $\varpi(\xi)\in \oplus H^{2j}_B(X)(j)$. 
 
 Let us denote by $\sG_{\iota_p}^\ast(X)$ the graded $\Q$-space of {\it $p$-adic De Rham-Betti classes}.

\medskip \noindent Around 1988, Fontaine asked questions of the following kind:
 
 \begin{qn} ($G_{\iota_p}(X)$)$\;$ Is the cycle class map $CH^\ast(X) \rightarrow \sG_{\iota_p}^\ast(X)$ surjective? \end{qn} 

  Stronger form: 
  {  $  tr. deg_\Q$ ($\iota_p$-periods of $X$) $\stackrel{?}{=} dim\, G_{mot}(X).$ } 

 \medskip\noindent 
Soon after, we gave a qualified answer: 
 {\it the answer depends on $\iota_p$.}

\medskip \begin{ex}\cite{A90} Let us take $\,p=11$ and $X$ be the modular curve $X_0(11)$. There are embeddings $\iota_p, \iota'_p$ for which 

$ tr. deg_\Q$ ($\iota_p$-periods of $X$) $= dim\, G_{mot}(X)= 4,$ and $G_{\iota_p}(X)$ has a positive answer,

$ tr. deg_\Q$ ($\iota'_p$-periods of $X$) $\leq  3,$ and $G_{\iota_p}(X)$ has a negative answer. \end{ex}

 \noindent This prompted Fontaine's next question: what if $\iota_p$ is sufficiently general?
 Our answer \cite{A90}, which involves the Tate conjecture, can now be expressed in a more transparent and general way, using motivic Galois groups:

 \begin{thm}\label{T3}  
            {Assume $\iota_p$ is ``sufficiently general". Then $T(X^n)$ for all $n\,$ implies that $G_{\iota_p}(X^n)$ has a positive answer for all $n\; $ (also in the strong form). \,}  \end{thm}
 
\begin{proof}
1) We have a period torsor $P(X)$ under $G_{mot}(X)_K$ defined over $K$ \cite[Ch. 7]{A04}, and $\varpi_{\iota_p}$ defines a $B_{dR, \hat K}$-point of $P(X)$. Note that for all $n>0$, there is a canonical isomorphism $P(X)\cong P(X^n)$ sending $\varpi_{\iota_p, X}$ to $\varpi_{\iota_p, X^n}$.

The tannakian category of motives (with coefficients in $\Q_p$) generated by $X$ is equivalent to $Rep_{\Q_p} G_{mot}(X)$, and its full subcategory of finite objets is equivalent to $Rep_{\Q_p} (G_{mot}(X)/G_{mot}(X)^0)$. There is a correponding map of torsors $P(X)\to P_{fin}(X)$. 

If $T(X^n)$ holds for all $n\,$, motivated cycles on $X^n$ are algebraic, and the finite motives are Artin motives. It is well-known that period torsors of Artin motives over $K$ are irreducible, hence $P_{fin}(X)$ is irreducible, with function field a finite extension $K'$ of $K$. 
 The generic fiber of  $P(X)\to P_{fin}(X)$ is a torsor under $G_{mot}(X)^0_{K'}$, hence irreducible. Therefore $P(X)$ is irreducible. 
  
2) The issue is whether $\varpi_{\iota_p, X}$ is {\it generic} in $P(X)$: indeed, if $\varpi_{\iota_p, X}$ is generic in $P(X)$, so is $\varpi_{\iota_p, X^n}$ in $P(X^n)\cong P(X)$, and $CH^\ast(X^n) \rightarrow \sG_{\iota_p}^\ast(X^n)$ is then surjective for any $n\,$ (and $  tr. deg_\Q$ ($\iota_p$-periods of $X$) $=  dim\, G_{mot}(X)$). 
    
 We fix an initial $\iota_p$; changing it into $\iota_p \circ \gamma$, for $\gamma \in G_K$, changes $\varpi_{\iota_p}$ into $\gamma \circ \varpi_{\iota_p}$.   
 Let $\hat L$ be a finite extension of $\hat K$ such that the Zariski closure of the image of $\varpi_{\iota_p}$ in $P(X)_{\hat K}$ has a $\hat L$-point $z$. We may identify $G_{mot}(X)_{\hat L}$ with $P(X)_{\hat L}: \, 1\mapsto z$. 
 
  Let $\{V_m\}$ be the countable set of $K$-subvarieties of $P(X)$ of lower dimension. Then $(V_{m})_{B_{dR, \hat L}}\circ \varpi_{\iota_p}^{-1}$ is an algebraic subvariety of $G_{mot}(X)_{B_{dR, \hat L}}$ of lower dimension. Its $\hat L$-points with respect to the $\hat L$-structure $G_{mot}(X)_{\hat L} $ are contained in $(V_{m})_{\hat L}\circ z^{-1}$. Thus its $\Q_p$-points with respect to the $\Q_p$-structure $G_{mot}(X)_{\Q_p}$
  are contained in the $\Q_p$-Zariski closure $W_m$ of $(V_{m})_{\hat L}\circ z^{-1}$, a subvariety of $G_{mot}(X)_{\Q_p}$ of lower dimension (so that the complement of $W_m(\Q_p)$ is dense open in $G_{mot}(X)(\Q_p)$). 
  
\smallskip   On the other hand, because $H^\ast(\bar X, \Q_p)$ is Hodge-Tate, an argument of Bogomolov \cite{Bo80} shows that the image $\Gamma_p$ of $ G_K \stackrel{\rho_p}{\to} GL(H^\ast(\bar X, \Q_p)) $ is open in the $\Q_p$-points of its Zariski closure. If $T(X^n)$ holds for all $n$, this Zariski-closure is $G_{mot}(X)_{\Q_p}$ (via Theorem \ref{T2}).  
 By Baire's category argument, $\Gamma_p$ is not contained in $\cup W_m(\Q_p)$: there exists $\gamma\in G_K$ whose image in $G_{mot}(X)(\Q_p)$ does not lie in any $W_m(\Q_p)$. Therefore $\gamma\circ \varpi_{\iota_p}$ is not contained in any $V_m$, hence is generic in the irreducible $K$-variety $P(X)$.  \end{proof}

  \begin{rem} In Theorem \ref{T3}, ``$\iota_p$ is sufficiently general" can be taken in a sense independent of $X$. 
  
  Indeed, $G_K$ is compact, hence Baire. The complement  $\Delta_{p, X}$ of $\cup W_m(\Q_p)$ in $\Gamma_{p }$ (with the notation of the above proof) is a dense $G_\delta$-set in $\Gamma_{p }$. By Bogomolov's openness result for $\rho_{X, p}$, $\rho_{X, p}^{-1}(\Delta_{p, X})$ is a dense $G_\delta$-set in $G_K$, and since there are countably many isomorphism classes of $X$'s,  $\Delta := \cap_{X}\, \rho_{X, p}^{-1}(\Delta_{p, X})$ is a dense $G_\delta$-set. Then for any $\gamma\in \Delta$, $\iota_p \circ \gamma$ is ``sufficiently general" in the sense of Theorem \ref{T3},  for all projective smooth varieties $X$ over $K$ simultaneously. \end{rem}

 \medskip {\it Acknowledgements.} I thank Bruno Kahn for his careful reading, and his corrections, questions and numerous comments which improved a lot the presentation.
  
     \end{sloppypar}
  
    \end{document}